\documentclass[12pt]{amsart}
\usepackage[T1]{fontenc}
\usepackage{tgtermes}
\usepackage{mathptmx}  
\usepackage[scaled=.92]{helvet}
\usepackage{amsthm,amsmath,amssymb}
\usepackage{mathrsfs}
\usepackage[numbers]{natbib}  
\usepackage[colorlinks,citecolor=blue,urlcolor=blue]{hyperref}  

\usepackage{enumerate}

\usepackage[all]{xy}

\newtheorem{theorem}{Theorem}[section]
\newtheorem{lemma}[theorem]{Lemma}
\newtheorem{proposition}[theorem]{Proposition}

\newtheorem{corollary}[theorem]{Corollary} 
\theoremstyle{definition}
\newtheorem{definition}[theorem]{Definition}

\newtheorem{notation}[theorem]{Notation} 
\theoremstyle{remark}
\newtheorem{remark}[theorem]{Remark}
\newtheorem{example}[theorem]{Example}
\newcommand{\U}{\mathcal{U}}
\newcommand{\dotminus}{ 
\buildrel\textstyle\ .\over{\hbox{ 
\vrule height3pt depth0pt width0pt}{\smash-} 
}\ }

\newcommand{\Ell}{\mathscr{L}}
\newcommand{\DEF}{\mathbf{Def}}

\newcommand{\R}{\mathbf{R}}
\renewcommand{\S}{\mathbf{S}}
\newcommand{\T}{\mathbf{T}}
\newcommand{\Metric}{\mathbf{Met}}
\newcommand{\Set}{\mathbf{Set}}
\newcommand{\Mod}{\operatorname{Mod}}
\newcommand{\eq}{\text{eq}}

\renewcommand{\SS}{\mathsf{S}}

\newcommand{\RR}{\mathsf{R}}
\newcommand{\UU}{\mathsf{U}}

\newcommand{\ff}{\mathsf{f}}

\newcommand{\dd}{\mathsf{d}}

\newcommand{\M}{\mathcal{M}}
\newcommand{\N}{\mathcal{N}}
\newcommand{\uu}{\mathbf{u}}

\newcommand{\rrto}{\ar[rr]}
\newcommand{\dto}{\ar[d]}

\renewcommand{\P}{\mathscr{P}}

\newcommand{\FF}{\mathcal{F}}
\newcommand{\GG}{\mathcal{G}}
\newcommand{\GGG}{\mathfrak{G}}
\newcommand{\FFF}{\mathfrak{F}}
\newcommand{\iii}{\mathfrak{i}}

\renewcommand{\And}{\wedge}
\newcommand{\Or}{\vee}
\newcommand{\Implies}{\Rightarrow}

\newcommand{\proves}{\vdash}

\renewcommand{\P}{\mathcal{P}}

\newcommand{\C}{\mathbf{C}}
\newcommand{\D}{\mathbf{D}}

\newcommand{\Lang}{\mathcal{L}}

\newcommand{\CLA}{\mathbf{CLA}}

\newcommand{\rto}{\ar[r]}

\newcommand{\urto}{\ar[ur]}

\newcommand{\drto}{\ar[dr]}

\newenvironment{cyclearray}[2][2pt]{\let\foo=\arraycolsep%
\setlength{\arraycolsep}{#1}\left(\kern-1pt \begin{array}{#2}}%
{\end{array}\kern-1pt\right)\setlength{\tabcolsep}{\foo}}

\newcommand{\restr}[2]{{#1}|_{#2}}

\newcommand{\Nat}{\mathbb{N}}

\newcommand{\DOM}{\mathscr{S}}

\newcommand{\var}[1]{{\mathtt{#1}}}
\newcommand{\tvar}[1]{{\overline{\var{#1}}}}

\newcommand{\Def}{:=}

  \title[Metric logical categories]{Metric logical categories and
  	conceptual completeness for first order continuous logic}

  \author{Jean-Martin Albert} 
   \address{Department of Mathematics and Statistics\\
     Towson University}
  \author{Bradd Hart}
  \address{Department of Mathematics and Statistics\\
    McMaster University\\
    Hamilton, Ontario}

\begin{document}
\begin{abstract}
We begin the study of categorical logic for continuous model theory.  In particular, we 
\begin{enumerate}
\item introduce the notions of metric logical categories and functors as categorical equivalents of a metric theory and interpretations,
\item prove a continuous version of conceptual completeness showing that $T^\eq$ is the maximal conservative expansion of $T$, and 
\item define the concept of a metric pre-topos.
\end{enumerate}
\end{abstract}

\maketitle

\section{Introduction}

In their seminal work on first order categorical logic \cite{first-order-categorical-logic}, Makkai and Reyes established several crucial results necessary for viewing logic in a categorical framework. First of all, they describe which properties a category should have in order to be able to carry out the usual definition of structure and satisfaction.  Consider a first-order language $\Ell$ which, for simplicity, only contains unary function and predicate symbols. The classical interpretation of $\Ell$ assigns a set $\M(\SS)$ to every sort of $\Ell$, a function $\M(\ff):\M(\SS)\to\M(\SS')$ to every function symbol $\ff:\SS\to \SS'$, and a subset $\M(\RR)\subseteq\M(\SS)$ to every predicate symbol $\RR$ on $\SS$. A term, which corresponds to the composite of finitely many function symbols, is interpreted by $\M(\ff_1\circ\cdots\circ\ff_n)=\M(\ff_1)\circ\cdots\circ\M(\ff_n)$, which shows that $\M$ behaves like a functor from a category generated by $\Ell$ to the category of sets. On the other hand, when an atomic formula $\varphi(\tvar{x})$ is intepreted, it gives rise to a subset $\M(\varphi(\tvar{x}))\subseteq\M(\SS_1)\times\cdots\times\M(\SS_n)$, and $\M(\varphi(\tvar{x}))$ is expressible categorically as a limit in $\Set$. This suggests that if a category $\C$ is sufficiently closed under limits, then we can interpret any first-order language in $\C$. In  \cite{first-order-categorical-logic}, Makka\"i and Reyes refer to these categories as {\em logical categories}.

To a logical category $\C$, Makka\"i and Reyes assign a first-order language $\Ell_\C$ and  first-order theory $T_\C$ to capture all the categorical structure of $\C$ as first-order statements. There is a canonical interpretation $\M_\C$ of $(\Ell_\C,T_\C)$ into $\C$ which is universal among all interpretations of first-order languages in $\C$. 

On the other hand, given a theory $T$ in some language $\Ell$, they showed how to construct a category, $\DEF(\Ell,T)$, of definable sets, as well as a canonical interpretation $\M_{(\Ell, T)}:(\Ell, T)\to\DEF(\Ell, T)$ which acts as a generic model of $T$. The correspondence $(\Ell, T) \mapsto \DEF(\Ell,T) = \C \mapsto (\Ell_\C, T_\C)$ satisfies the relationship $$\Mod(\Ell,T) \cong \hom(\DEF(\Ell,T), \Set)\cong \Mod(\Ell_\C,T_\C)$$ an equivalence of categories where $\hom(\DEF(\Ell,T), \Set)$ represents the class of all functors which preserve the logical structure.
%


An important step along the way is the identification of the role of $T^{\eq}$.  They show that $T^{\eq}$ is the maximal conservative extension of $T$ i.e. if $T'$ is a conservative extension of $T$ then $T'$ is interpretable in $T^{\eq}$.  This result is known as conceptual completeness - imaginaries are viewed as concepts associated to $T$ and any other abstract concepts can be realized as concrete imaginaries.  
They also link the concept of logical category to that of a {\em pre-topos}, which was introduced by Grothendieck in \cite{sga4}. Any logical category can be completed to a pre-topos, and the process of completion mirrors the construction of $T^\eq$ quite closely.
%
With this machinery they prove the correspondence $$\C^{\eq} \cong \DEF(\Ell_\C^{\eq},T_\C^{\eq})$$ where $\C$ is a logical category, $\C^\eq$ is the pre-topos completion of $\C$, and $(\Ell_\C,T_\C)$ is the first order language and theory associated to $\C$. This established an equivalence between pre-topoi and first-order theories with elimination of imaginaries.

Our goal is to prove similar results in the continuous setting.  There are several hurdles to overcome.  First of all, the syntax for continuous logic involves more data, in particular the moduli of continuity, and this must be incorporated.  We attack these problems in sections \ref{sec:continuous-languages-and-structures} through \ref{sec:metric-logical-category}.  We create the notion of a metric logical category as a category with additional structure and define the notion of metric logical functor.  These two notions are meant to be the categorical equivalents of metric theories and structures.  The culmination of this work is Theorem \ref{thm:models-are-equivalent-to-logical-functors}.
Second, a conceptual completeness theorem in the continuous setting must be proved, and is, as Theorem \ref{thm:weak-conceptual-completeness}.   Third, we tie the notion of metric logical category and conceptual completeness together by introducing the appropriate notion of a metric pre-topos.  For us these will be metric logical categories that are maximal with respect to axiomatizable limits and co-limits. The main results regarding metric pre-toposes and their connection with conceptual completeness are Theorem \ref{thm:T-eq-is-a-pre-topos} and Corollary \ref{thm:pre-topos-equiv-of-cats}.

\section{$[0,1]$-valued Languages and Structures}\label{sec:continuous-languages-and-structures}

In this paper we will use the standard continuous logic as described in \cite{continuous-first-order-logic} and \cite{model-theory-for-metric-structures}. We introduce the logic somewhat slowly in order to highlight the categorical approach. 

\subsection*{$[0,1]$-valued Languages}

We will first introduce a $[0,1]$-valued logic without any metric symbols and introduce the metric only later. A continuous language will consist of a set of sorts, a set of relation symbols, and a set of function symbols. All functions and relations will be assumed to be finitary, and we will think of $0$-ary function symbols as constants. For every sort $\SS$ of $\Ell$, we have an unlimited supply of {\em variables of sort $\SS$}, ranging over symbols of the form $\var{x}$, $\var{y}$ and $\var{z}$. The sort of a variable $\var{x}$ will be considered a property of $\var{x}$, and denoted $\DOM(\var{x})$. If $\tvar{x}=(\var{x}_i: i\in I)$ is an $I$-indexed sequence of variables, then we define $\DOM(\tvar{x})=(\DOM(\var{x}_i): i\in K)$, where $K$ is the largest subset of $I$ with the property that $\var{x}_k\not=\var{x}_\ell$ whenever $k,\ell\in K$.

Terms in $[0,1]$-valued logic are constructed as in first-order logic, and consist of finite compositions of function symbols and variables of the appropriate types. We write a term $t$ as $t(\tvar{x})$ to indicate that the free variables of $t$ are among the variables listed in $\tvar{x}$, and the target sort of $t(\tvar{x})$ will be denoted $\DOM(t(\tvar{x}))$.

A \emph{(continuous) $n$-ary logical connective} for $n \in \Nat$ is a formal symbol $\uu$ corresponding to a continuous function $u:[0,1]^{n}\to[0,1]$. An atomic  formula is an expression of the form $\varphi=\RR(t_1(\tvar{x}_1),...,t_n(\tvar{x}_n))$, where $\RR\subseteq \SS_1\times\cdots\times\SS_n$ is a relation symbol, and $\DOM(t_i(\tvar{x}))=\SS_i$. If $\uu$ is an $n$-ary connective, and $(\varphi_i:i<n)$ is a sequence of formulas then $\uu(\varphi_i(\tvar{x}):i< n)$ is a formula. Finally, if $\varphi(\tvar{x}, \tvar{y})$ is a formula, where $\tvar{x}$ is a finite tuple of variables, then $\forall_\tvar{x}[\varphi(\tvar{x}, \tvar{y})]$ and $\exists_\tvar{x}[\varphi(\tvar{x}, \tvar{y})]$ are formulas as well. To make continuous formulas easier to read, we use the shorthands $\varphi\And\psi:=\max\{\varphi, \psi\}$ and $\varphi\Or\psi:=\min\{\varphi, \psi\}$. Again, we write $\varphi(\tvar{x})$ to indicate that the free variables of $\varphi$ are among the variables listed in $\tvar{x}$.
A sentence is a formula with no free variables and a theory is a set of sentences.

\subsection*{$[0,1]$-valued Structures}

A {\em $[0,1]$-valued $\Ell$-structure} consists of an assignment of a set $\M(\SS)$ to every sort symbol of $\Ell$, along with a function $\M(\ff):\prod\M(\SS_i)\to\M(\SS)$ for every function symbol $\ff:\SS_1\times\cdots\times\SS_n\to\SS$, and for every relation symbol $\RR\subseteq \SS_1\times\cdots\times\SS_n$, a function $\M(\RR):\prod\M(\SS_i)\to [0,1]$. 

\begin{notation}
Throughout this paper, we will be using the notation $\M:\Ell\to\Set$ to denote an $\Ell$-structure.
\end{notation}

If $\var{x}$ is a variable, then we will write $\M(\var{x})$ for the set $\M(\DOM(\var{x}))$, and if $\tvar{x}=(\var{x}_i:i\in I)$ is an $I$-indexed tuple of variables, then we define $\M(\tvar{x})=\prod_{i\in I}\M(\var{x}_i)$.  Terms, atomic and quantifier-free formulas in continuous logic are interpreted simply by composing the various interpretations of the symbols in them. For quantified formulas, we define $$\M(\forall_{\tvar{y}}\varphi(\tvar{x}, \tvar{y}))(\overline{x})=\sup\{\M(\varphi(\tvar{x}, \tvar{y})))(\overline{x},\overline{y}):\overline{y}\in \M(\tvar{y})\}$$ and $$\M(\exists_{\tvar{y}}\varphi(\tvar{x}, \tvar{y}))(\overline{x})=\inf\{\M(\varphi(\tvar{x}, \tvar{y})))(\overline{x},\overline{y}):\overline{y}\in \M(\tvar{y})\}$$


\begin{notation}
The interpretation of the term $t(\tvar{x})$ and formulas $\varphi(\tvar{x})$ will be denoted by $\M(t(\tvar{x}))$ and $\M(\varphi(\tvar{x}))$ respectively. From the definition, we see that for every term $t(\tvar{x})$ of $\Ell$, $\M(t(\tvar{x}))$ is a function, $\M(\tvar{x})\to\M(\DOM(t(\tvar{x})))$, and for every formula $\varphi(\tvar{x})$ of $\Ell$, $\M(\varphi)$ is a function $\M(\tvar{x})\to [0,1]$.
\end{notation}

\subsection*{Satisfaction} As is now traditional in continuous logic, we think of 0 as nominally true. So if $\sigma(\tvar{x})$ is an $\Ell$-formula,
$\M$ is an $\Ell$-structure, and $\overline{a}$ is a tuple of elements of $\M$ from sorts matching $\DOM(\tvar{x})$, then we write $\M\models\sigma(\overline{a})$ if $\M(\varphi(\tvar{x}))(\overline{a}) = 0$. 

We will write $\M\models\sigma(\tvar{x})$ to mean $\M\models\sigma(\overline{a})$ for every $\overline{a}\in\M$; notice that this is equivalent to $\M \models \sup_\tvar{x} \sigma(\tvar{x})$.

The function $x \dotminus y := \max\{x - y, 0 \}$ is often useful for formally expressing formulas in our logic, we will adopt some shorthand notation to make things more readable.  For instance, $\M \models \varphi \leq r$ will mean $\M \models \varphi \dotminus r$ and we make a similar convention for $\geq$.  

If $\M,\N:\Ell\to\Set$ are structures, then a map $h:\M\to\N$ is a collection $\{h_\SS:\SS\text{ is a sort of }\Ell\}$ of functions $h_\SS:\M(\SS)\to\N(\SS)$ which preserve the values of all function and relation symbols of $\Ell$; $h$ is elementary if it preserves the value of all formulas. The class of all structures $\M:\Ell\to\Set$ which are models of a theory $T$ will be denoted $\Mod(\Ell, T)$. We view $\Mod(\Ell, T)$ as a category whose morphisms are elementary maps. 

\subsection*{Metric Theories and Metric Structures}

\begin{definition}\label{def:pseudo-metric}
A formula $\varphi(\tvar{x},\tvar{y})$, where $\tvar{x}$ and $\tvar{y}$ are tuples of variables of the same length and sort, is a pseudo-metric for a theory $T$ if 
\begin{enumerate}
\item $T\models \varphi(\tvar{x},\tvar{x})=0$,
\item $T\models |\varphi(\tvar{x},\tvar{y})-\varphi(\tvar{y},\tvar{x})|=0$	, and
\item $T\models (\varphi(\tvar{x},\tvar{y})+\varphi(\tvar{y},\tvar{z}))\geq\varphi(\tvar{x},\tvar{z})$.
\end{enumerate}
\end{definition}

\begin{definition}\label{def:metric-language-and-theory}
A theory $T$ in a language $\Ell$ is a {\em metric theory} if  for every sort $S$ of $\Ell$ we associate a pseudo-metric $d_S(x,y)$ for the theory $T$, where $x$ and $y$ are variables of sort $S$. Moreover, for every function symbol $\ff$, and every relation symbol $\RR$ we have:
\begin{enumerate}
\item For every $\varepsilon>0$, there are $\delta_1,...,\delta_n>0$ such that 
\[
T\models\{\dd_{\var{x}_i}(\var{x}_i, \var{y}_i)<\delta_i:1\leq i\leq n\}\Implies \dd_\SS(\ff(\tvar{x}),\ff(\tvar{y}))\leq\varepsilon.
\]
\item For every $\varepsilon>0$, there are $\delta_1,...,\delta_n>0$,
\[
T\models \{\dd_{\var{x}_i}(\var{x}_i, \var{y}_i) <\delta_i:1\leq i\leq n\}\Implies |\RR(\tvar{x})-\RR(\tvar{y})|\leq\varepsilon.
\]
\end{enumerate}
\end{definition}
Of course formally by the expressions above we mean 
\[
T \models \max\{\{\dd_{\var{x_i}}(\var{x}_i,\var{y}_i) \geq \delta_i: i \leq n\}, \dd_\SS(\ff(\tvar{x}),\ff(\tvar{y}))\leq\varepsilon\}
\]
and something similar for relations but the sequent notation is more readable.

Let $T$ be a metric theory in a language $\Ell$ and $\M$ an $\Ell$-structure satisfying $T$. Then for every $\SS$, $\M(\SS)$ is a pseudo-metric space with pseudo-metric $\M(\dd_\SS)$. Moreover, the interpretation of every function and relation symbol in $\M$ is uniformly continuous with respect to the assigned metrics on the domain and range of said symbols. If $\M(\dd_\SS)$ is a complete {\em metric} on $\M(\SS)$ for every $\SS$, then $\M$ will be called a {\em model} of $T$. 


Suppose $T$ is a metric theory. If $\tvar{x}=\var{x}_1\var{x}_2\cdots\var{x}_n$ and $\tvar{y}=\var{y}_1\cdots\var{y}_n$ are two sequences of variables such that $\DOM(\var{x}_i)=\DOM(\var{y}_i)$ for every $i$, we define a pseudo-metric  $D$ on $\DOM(\tvar{x})$ as follows: 
\[
D(\tvar{x},\tvar{y}) := \max\{ \dd_{\var{x_i}}(\var{x}_i,\var{y}_i) : i \leq n\}.
\] 
We will write $D_\tvar{x}$ or $D_{\DOM(\tvar{x})}$ to mean that we are talking about the pseudo-metric on the tuple $\DOM(\tvar{x})$. There are many possible choices of metrics on $\DOM(\tvar{x})$. However, any two definable pseudo-metrics yield the same continuous functions.

\begin{theorem}\label{thm:formulas-are-uniformly-continuous} Suppose that $T$ is a metric theory and $M$ is a model of $T$.
If $\varphi(\tvar{x})$ is a formula of $\Ell$, then $\M(\varphi(\tvar{x}))$ is uniformly continuous with respect to $D_{\tvar{x}}$, and if $t(\tvar{x})$ is a term of $\Ell$, then $\M(t(\tvar{x}))$ is uniformly continuous with respect to $D_{\tvar{x}}$ and $D_{\DOM(t)}$.
\end{theorem}

\begin{definition}\mbox{}
\begin{enumerate}
\item By $\Metric$ we will denote the category of metric spaces with metrics bounded by 1 and with uniformly continuous maps as morphisms.
\item A metric model of $T$ will be denoted by an arrow $\M:(\Ell, T)\to\Metric$, and the underlying structure will be denoted by an arrow $\M:\Ell\to\Set$. The category of metric models $\M:(\Ell, T)\to\Metric$ will be denoted $\Mod^*(\Ell, T)$; note that the notion of elementary map does not formally change but now must respect the specified metrics on the sorts.
\end{enumerate}
\end{definition}

\begin{theorem}[\cite{continuous-first-order-logic}, Prop. 2.10)]
\label{thm:completion-of-metric-pre-structure}Let $(\Ell, T)$ be a metric theory, and suppose $\M:\Ell\to\Set$ is a model of $T$. Then there is a model $\M^*:(\Ell, T)\to\Metric$ such that for every sentence $\sigma$ in $\Ell$, $\M\models \sigma$ if and only if $\M^*\models \sigma$. Furthermore, if $h:\M\to\N$ is an elementary map, then there is an elementary map $h^*:\M^*\to\N^*$. This defines an pair of adjoints $$\xymatrix{{\Mod(\Ell, T)\ar@/^{1pc}/[rr]^{(-)^*}}&{\perp}&{\Mod^*(\Ell, T)\ar@/^{1pc}/[ll]^F}}$$ where $F:\Mod^*(\Ell, T)\to\Mod(\Ell, T)$ represents the forgetful functor.
%
%
\end{theorem}



\section{Continuous Syntactic Categories}


We now give the definition of continuous syntactic categories. The approach we take to logical categories in the continuous context is somewhat related to the concept of hyperdoctrine which is described for classical and intuitionistic logic in  \cite{adjointness-in-foundations} and \cite{hyperdoctrines}. We begin by describing a continuous analogue of boolean algebras, which we will later use to interpret $[0,1]$-valued formulas. For every $n\in\Nat$, consider the set $U_n$ of all continuous functions $u:[0,1]^n\to[0,1]$.  We consider a language which has, for every $n$, an $n$-ary function symbol $u$ for every $u \in U_n$.  For us, the basic model in this language is $[0,1]$ itself with the natural interpretation of each $u$.
A {\em continuous logical algebra} is an algebra that satisfies the equational theory in this language of the standard structure on $[0,1]$. The category of all continuous logical algebras together with homomorphisms between them will be denoted $\CLA$.
We make a few observations and notational conventions when working with continuous logical algebras.

First of all, there is a natural lattice structure on a continuous logical algebra since among the functions in $U_2$ one has $\min$ and $\max$.  This is part of the equational theory of $[0,1]$.  We will use the order notation to denote the lattice ordering so $x \leq y$ means $x = \min(x,y)$ for instance. When $n = 0$, all $u \in U_0$ are just constants or names for the elements of $[0,1]$.
The ordering induced on these constants is exactly the same as in $[0,1]$ and in fact, 0 is the least element of our lattice and 1 is the maximal element.  Since every continuous logical algebra will interpret the constant $\varepsilon$ for every $\varepsilon \in [0,1]$, we will say that a continuous logical algebra $A$ is {\em standard} if whenever $a \in A$ and $a \leq \varepsilon$ for every $\epsilon > 0$ then $a = 0$.

The most important use of the concept of continuous logical algebra is the following. Suppose that $\M$ is a $[0,1]$-valued structure and $\SS$ is a sort.  Consider $\Ell^\M(\tvar{x}))$ to be the collection of all the $[0,1]$-valued interpretations of formulas $\varphi(\tvar{x})$ where $\tvar{x}$ is a variable of sort $\SS$.  Every continuous connective $u \in U_n$ has a natural interpretation on this set via composition and so it remains to see that this algebra is a continuous logical algebra.  Every $\M(\varphi(\tvar{x}))\in \Ell^\M(\tvar{x}))$ is a function from $\M(\SS )$ to $[0,1]$ and so can be thought of as an element of $[0,1]^{\M(\SS)}$ which is a product of the continuous logical algebra structure on $[0,1]$.  It is easy to see then that $\Ell^\M(\tvar{x}))$ is a subalgebra of this product and hence a continuous logical algebra.

%

\subsection*{Quantification}

In the 1960's, F.W. Lawvere realized that quantification in first-order logic can be understood in purely categorical terms using adjoint functors. More specifically, if we denote by $F(\tvar{x})$ the set of all formulas in a language $\Ell$ whose free variables are among those listed in the tuple $\tvar{x}$, then $F(\tvar{x})$ is partially ordered by entailment, i.e. $\varphi(\tvar{x})\leq\psi(\tvar{x})$ if and only if $\varphi(\tvar{x})\proves\psi(\tvar{x})$. The inclusion map $\iota:F(\tvar{x})\to F(\tvar{x}\tvar{y})$, viewed as a functor, has both a left and a right adjoint: $\exists_{\tvar{y}}\dashv\iota\dashv\forall_\tvar{y}$. The adjunction properties mirror the natural deduction rules for introducing and eliminating the quantifiers. 

In the continuous case, we denote by $\Lang(\tvar{x})$ the set of all formulas whose free variables are among $\tvar{x}$. 
For any set of variables $\tvar{x}$, we have the continuous logical algebra $\Lang^\M(\tvar{x})=\{\M(\varphi(\tvar{x})):\varphi\in\Lang(\tvar{x})\}$ and we can consider it a simple category with arrows given by the lattice ordering. Suppose that $\iota^\M:\Lang^\M(\tvar{x})\to\Lang^\M(\tvar{x}\tvar{y})$ is the inclusion map and note that the continuous quantifiers  $\forall_\tvar{y}^\M$ and  $\exists_\tvar{y}^\M$ define maps from
$\Lang^\M(\tvar{x}\tvar{y})\to\Lang^\M(\tvar{x})$.
  Then for every $\M$, we get an adjunction $\forall^\M_\tvar{y}\dashv\iota^\M\dashv\exists^\M_\tvar{y}$. Note how the left and the right adjoints are reversed because of the convention that $0$ is true.  This observation was made in \cite{continuous-first-order-logic} and justifies the definition of the quantifiers in the continuous context.

\begin{definition}
A {\em continuous syntactic category} is a 
category $\C$  with all finite products and a contravariant functor $\Lang:\C \to \CLA$.  $\Lang$ satisfies the following:
\begin{enumerate}
\item If 1 is the empty product then $\Lang(1)$ is a standard continuous logical algebra, and
\item for every $A,B\in\C$, and projection map $\pi:A\times B\to B$, the map 
\[
\Lang(\pi):\Lang(B)\to\Lang(A\times B)
\]
 is a continuous logical algebra embedding.
Moreover, $\Lang(\pi)$ has a left adjoint $\forall_\pi:\Lang(A\times B) \to \Lang(B)$ which preserves the constants and the order. That is, for every $g \in \Lang(A\times B)$ and $h \in \Lang(B)$
\[
\forall_\pi(g)  \leq h \mbox{ iff } g \leq  \Lang(\pi)(h).
\]
\end{enumerate}
\end{definition}

\begin{remark}As a justification for calling the above map $\forall_\pi$ an adjoint, we think of a continuous logical algebra as a category with the algebras elements as objects and a unique arrow between $x$ and $y$ iff $x \leq y$.  With this identification, $\forall_\pi$ is a functor which is a left adjoint to the functor $\Lang(\pi)$.

We note that we also have a right adjoint when $\pi$ is as above and $\forall_\pi$ is the left adjoint.  Define $\exists_\pi(g) := 1 \dotminus \forall_\pi (1 \dotminus g)$ and it is easy to show that
\[
g \leq \exists_\pi(h) \mbox{ iff } \Lang_\pi(g) \leq  h
\]
\end{remark}

An example is in order.  $\Metric$ can be  considered a continuous syntactic category.  $\Lang(X)$ is the continuous logical algebra of uniformly continuous functions from $X$ to $[0,1]$; it is clear that all functions from a one point set to $[0,1]$ is isomorphic to $[0,1]$ so we have $\Lang(1) \cong [0,1]$.  If $\pi:X\times Y \to Y$ in $\Metric$ is the projection map onto $Y$ then $\Lang(\pi)$ is just composition by $\pi$ and it is easily seen to be an continuous logical embedding. Define $\forall_\pi(g)$ for any $g \in \Lang(X)$ by:
\[ \forall_\pi(g)(y) = \sup\{ g(a) : \pi(a) = y\} \]
for all $y \in Y$.  One checks that $\forall_\pi(g) \leq h$ iff $g \leq h \circ \pi$ which shows that $\forall_\pi$ is the left adjoint.  

A continuous syntactic category has enough structure to interpret a continuous language. Let $\Ell$ be a continuous language, and $\C$ be a continuous syntactic category. A {\em $\C$-structure $\M$ of type $\Ell$}, which we will denote by $\M:\Ell\to\C$,  consists of the following data:
\begin{enumerate}
\item For every sort symbol $\SS$ of $\Ell$, an object $\M(\SS)$ of $\C$;
\item For every relation symbol $\RR\subseteq\SS_1\times\cdots\times\SS_n$, $\M(\RR)$ is an element of $\Lang(\M(\SS_1)\times\cdots\times\M(\SS_n))$.
\item For every function symbol $\ff:\SS_1\times\cdots\times\SS_n\to\SS$, $\M(\ff)$ is a morphism $\M(\SS_1)\times\cdots\times\M(\SS_n)\to \M(\SS)$
\end{enumerate}

In order to simplify the notation, as before, if $\tvar{x}$ is a finite tuple of variables, we write $\M(\tvar{x})$ instead of $\prod \M(\DOM(\var{x}_i))$. 

\paragraph{\bf Terms} We can provide interpretations for terms in $\Ell$ as follows: 
\begin{enumerate}
\item If $\var{x}$ is a variable, then $\M(\var{x})=Id_{\M(\DOM(\SS))}:\M(\DOM(\var{x}))\to\M(\DOM(\var{x}))$. 
\item If $\ff(\var{x}_1,...,\var{x}_n)$ is a function symbol of type $\SS$, and for every $i$, $s_i$ is a term of sort $\DOM(\var{x}_i)$, then $$\M(f(s_1,...,s_n))=\M(\ff)(\M(s_1),...,\M(s_n))$$
\end{enumerate}

\paragraph{\bf Atomic Formulas} 
If $\varphi(\tvar{x})=\RR(t_1(\tvar{x}_1),...,t_n(\tvar{x}_n))$, then by definition $\M(\RR)\in\Lang(\tvar{y})$, and $\M(t_i(\tvar{x}_i)):\M(\tvar{x}_i)\to\M(\var{y}_i)$ for every $i$, so that we get the tuple 
\[
(\M(t_1(\tvar{x}_1)),...,\M(t_n(\tvar{x}_n))): \prod_i\M(\tvar{x_i})\to\M(\tvar{y}).
\]
 and so
 \[
 \Lang(\M(t_1(\tvar{x}_1)),...,\M(t_n(\tvar{x}_n))):\Lang(\M(\tvar{y})) \to \Lang(\prod_i\M(\tvar{x_i}))
 \]
 we may therefore define 
 \[
 \M(\RR(t_1(\tvar{x}_1),...,t_n(\tvar{x}_n)))=\Lang(\M(t_1(\tvar{x}_1)),...,\M(t_n(\tvar{x}_n)))(\M(\RR)).
 \]

\paragraph{\bf Connectives} Suppose that $\varphi_i(\tvar{x})$ is a formula for every $i \leq n$ and $\uu$ is an $n$-ary connective corresponding to the continuous function $\uu:[0,1]^n\to[0,1]$. We have that $\M(\varphi_i(\tvar{x}) \in \Lang(\tvar{x})$ for every $i$ and so we define 
\[
\M(\uu(\varphi_i(\tvar{x}):i\leq n))=\uu(\M(\varphi_i(\tvar{x})):i\leq n).
\]
Notice that $\uu$ on the right-hand side of this equation is the interpretation of $\uu$ in the continuous logical algebra $\Lang(\tvar{x})$.

\paragraph{\bf Quantifiers} Finally, if $\varphi(\tvar{x},\tvar{y})$ is a formula, where $\tvar{x}$ and $\tvar{y}$ are disjoint {\em finite} tuples of variables, consider the projection map $\pi := \pi_\tvar{y}:\M(\tvar{x}\tvar{y})\to\M(\tvar{y})$; remember that $\M(\tvar{x}\tvar{y}) = \M(\tvar{x}) \times \M(\tvar{y})$.
By definition $\M(\varphi(\tvar{x},\tvar{y}))$ is an element of $\Lang(\tvar{x}\tvar{y})$, and so we define 
\[
\M(\forall_{\tvar{x}}(\varphi(\tvar{x},\tvar{y})))=\forall_\pi(\M(\varphi(\tvar{x}, \tvar{y})))
\]
 and for convenience
 \[
 \M(\exists_{\tvar{x}}(\varphi(\tvar{x},\tvar{y})))=\exists_\pi(\M(\varphi(\tvar{x}, \tvar{y})))
 \]



\subsection*{Truth and Satisfaction in Continuous Syntactic Categories}

If $\C$ is a continuous syntactic category, and $\M:\Ell\to\C$ is a structure then 
for any formula $\varphi(\tvar{x})$ we will write $\M \models \varphi(\tvar{x})$ iff $\M(\varphi(\tvar{x})) = 0$ in $\Lang^\M(\tvar{x})$. 
We similarly adopt our notation regarding the use of inequalities and sequents from earlier.

\begin{definition}\mbox{}
\begin{enumerate}
\item For an $\Ell$-theory $T$, a $\C$-model is a structure $\M:\Ell\to\C$ with the property that $\M\models \varphi$ whenever $\varphi \in T$. A model of $T$ in $\C$ will be denoted by an arrow $\M:(\Ell, T)\to \C$. 
\item If $\M$ and $\N$ are both models of $T$ in $\C$ a collection 
\[
f =  \{ f_S:\M(S) \to \N(S) : \mbox{ $S$ is a sort in $\Ell$}\}
\]
 is an elementary map from $\M$ to $\N$ if for every formula $\varphi(\tvar{x})$,
\[
\M(\varphi(\tvar{x})) = \Lang(f_S)(\N(\varphi(\tvar{x})))
\]
where $S$ is the sort of the variables $\tvar{x}$.
\item The category of all $\C$-models of $T$ is denoted $\Mod_\C(\Ell, T)$ with elementary maps as arrows.  
\end{enumerate}
\end{definition}

If $T$ is an $\Ell$-theory, then for any formula $\varphi(\tvar{x})$, we write $T\models_\C\varphi(\tvar{x})$ to mean that every structure $\M:\Ell\to\C$ which satisfies $T$ also satisfies $\varphi(\tvar{x})$.

\section{Definability and the Category of Definable Sets}

For the remainder of this paper, we concentrate on metric languages and theories, and fix such an $(\Ell, T)$. Definition \ref{def:quantification-over-conceptually-definable-sets} is very general, and says that a set is definable if and only if it is possible to quantify over the elements of that set. One of the main obstacles in porting the notions of categorical logic to the continuous framework is that continuous logic lacks the correspondence between definable sets and formulas. 
An important tool for the remainder of this paper is Theorem \ref{thm:syntactic-definability}, which establishes a syntactic criterion for determining whether a set is definable. 

\subsection*{Definable Sets and Definable Functions}
Suppose that $T$ is a metric theory, $I$ is a finite set, $(\SS_i:i\in I)$ is an $I$-indexed list of sorts of $\Ell$ and $\tvar{x}$ is an $I$-indexed sequence of variables $\DOM(\var{x}_i) = \SS_i$ for every $i$ . Remember that if $\varphi(\tvar{x})$ is a formula and $M$ is a model of $T$ then $\M(\varphi(\tvar{x})):\prod_{i\in I}\M(\SS_i)\to [0,1]$.  We say that a sequence of formulas $f = \langle \varphi_n(\tvar{x} : n \in \Nat \rangle$ is a $T$-definable predicate if for every model of $T$,
$\M:\Ell\to\Metric$  the sequence of functions $\M(\varphi_n(\tvar{x}))$ converges uniformly to a function $\M(f):\prod_{i\in I}\M(\SS_i)\to [0,1]$. The fact that the sequence $\M(\varphi_n(\tvar{x}))$ is uniform can easily be expressed by first-order statements. The interpretation of a formula is always a definable predicate.

\begin{definition} For a metric theory $T$ and a $T$-definable predicate $f$, \\
if $\M(f):\prod_{i\in I}\M(\SS_i)\to[0,1]$ then the {\em zero-set} of $f$ is the set 
\[
Z_\M(f)=\{\overline{x}\in\prod_{i\in I}\M(\SS_i):f(\overline{x})=0\}.
\]
\end{definition}


\begin{definition}\label{def:quantification-over-conceptually-definable-sets}
\label{prop:quantification-over-definable-sets} For a metric theory $T$,
suppose $\tvar{x}$ and $\tvar{y}$ are finite tuples of variables and $\psi(\tvar{x})$ is a $T$-definable predicate. Then $\psi(\tvar{x})$ is a $T$-definable set if and only if for every formula $\varphi(\tvar{x}, \tvar{y})$, there are $\Ell$-formulas $\exists_{\psi(\tvar{x})=0}[\varphi(\tvar{x}, \tvar{y})]$ and $\forall_{\psi(\tvar{x})=0}[\varphi(\tvar{x}, \tvar{y})]$ such that, in any model $\M:(\Ell, T)\to\Metric$. 
\begin{enumerate}
\item $\M(\exists_{\psi(\tvar{x})=0}[\varphi(\tvar{x}, \tvar{y})])=\inf\{\varphi(\overline{x}, \overline{y}):\overline{x}\in Z_\M(\psi(\tvar{x}))\}$
\item $\M(\forall_{\psi(\tvar{x})=0}[\varphi(\tvar{x}, \tvar{y})])=\sup\{\varphi(\overline{x}, \overline{y}):\overline{x}\in Z_\M(\psi(\tvar{x}))\}$
\end{enumerate}
\end{definition}




\begin{lemma}[Proposition 9.19 in \cite{model-theory-for-metric-structures}]\label{thm:definable-if-bounded-by-formula}
Let $\tvar{x}$ be a finite tuple of variables. A definable predicate $A(\tvar{x})$ is a $T$-definable set if and only if there is a $T$-definable predicate $\varphi(\tvar{x})$ such that in every $\M\in\Mod^*(\Ell, T)$, we have $\varphi(\overline{x})=0$ for every $\overline{x}\in Z_\M(A(\tvar{x}))$, and $D(\overline{x}, A)\leq\varphi(\overline{x})$ for every $\overline{x}\in\M(\tvar{x})$.
\end{lemma}


\begin{theorem}[Theorem 9.2, \cite{model-theory-for-metric-structures}]\label{thm:syntactic-definability}
Let $(\Ell, T)$ be a metric language and $\varphi(\tvar{x})$ a formula.  Suppose 
\begin{enumerate}
\item $T \models \forall_\tvar{x}\exists_\tvar{y} \max\{ \varphi(\tvar{y}), |\varphi(\tvar{x}) - d(\tvar{x},\tvar{y})| \}$ and
\item $T \models \forall_\tvar{x} | \varphi(\tvar{x}) - \exists_\tvar{y}(\varphi(\tvar{y}) + d(\tvar{x},\tvar{y})) |$.
\end{enumerate}
Then for any model $\M$ of $T$, $\M(\varphi)(\bar x) = d(\bar x, D)$ where $D = \{ \bar x \in \M : \M(\varphi)(\bar x) = 0 \}$.  In particular, $\varphi(\tvar{x})$ is a $T$-definable set.
\end{theorem}

\begin{definition}\label{def:syntactic-definability}
From now on, when we write that $\varphi(\tvar{x})$ is a $T$-definable set, we will mean that $\varphi(\tvar{x})$ satisfies the conditions of Theorem \ref{thm:syntactic-definability}, and thus defines the distance to its zero-set in any model.
\end{definition}

\begin{definition}\label{def:definable-function}
Let $A(\tvar{x})$ and $B(\tvar{y})$ be $T$-definable sets. A {\em $T$-definable function from $A(\tvar{x})$ to $B(\tvar{x})$} is a $T$-definable predicate $\alpha(\tvar{x},\tvar{y})$ such that $Z_\M(\alpha)$ is the graph of a function from $Z_\M(A)$ to $Z_\M(B)$ for all  models $\M\models T$.
\end{definition}

\begin{theorem}\label{thm:continuity-of-definable-functions}
Let $\alpha(\tvar{x},\tvar{y})$ be a definable function $A(\tvar{x})\to B(\tvar{y})$, and let $f_\alpha$ be the induced function on a model $\M$ of $T$.  Then  $f_\alpha:Z_\M(A(\tvar{x}))\to Z_\M(B(\tvar{y}))$ is uniformly continuous with respect to the metric $D_\tvar{x}$ on $\M(\tvar{x})$ and $D_\tvar{y}$ on $\M(\tvar{y})$
\end{theorem}


\subsection*{The Category of Definable Sets}

We now come to the definition of the central category of this paper: the category of definable sets of a metric theory.

\begin{definition}\label{def:def-L-T}
To a metric theory $(\Ell, T)$ we associate the category $\DEF(\Ell, T)$ defined as follows: the objects of $\DEF(\Ell, T)$ are  formulas  $A(\tvar{x})$ which are $T$-definable. A morphism $\alpha:A(\tvar{x})\to B(\tvar{y})$ is an equivalence class of the form $[\alpha(\tvar{x}, \tvar{y})]$ where $\alpha(\tvar{x},\tvar{y})$ is a definable function $A(\tvar{x})\to B(\tvar{y})$ and 
$\alpha(\tvar{x},\tvar{y})\sim \beta(\tvar{x},\tvar{y})$ if they define the same function from $Z_\M(A)$ to $Z_\M(B)$ for all models $\M$ of $T$.

If $\alpha(\tvar{x}, \tvar{y})$ and $\beta(\tvar{y}, \tvar{z})$ are definable functions then 
the composition of $[\alpha]$ and $[\beta]$ is the equivalence class of the formula $\exists_{\tvar{y}}(\alpha(\tvar{x}, \tvar{y}) \wedge \beta(\tvar{y}, \tvar{z}))$. It is straightforward to show that this definition of composition is well-defined.
\end{definition}

\section{The Internal Language of a Continuous Syntactic Category}

Fix a continuous syntactic category $\C$. We produce a language $\Ell_\C$, a theory $T_\C$, and a model $\M_\C:(\Ell_\C,T_\C)\to\C$ such that under suitable assumptions placed on $\C$, $(\Ell_\C,T_\C)$ and $\M_\C$ are universal among those languages interpretable in $\C$.

\begin{definition}\label{def:the-language-of-a-logical-category}
The canonical language for a continuous syntactic category $\C$ is the continuous language $\Ell_{\C}$ is defined as follows:
\begin{enumerate}
\item A sort symbol $\SS_{A}$ for every object $A\in\C$; 
\item A relation symbol $\RR_{\varphi}$ with domain $\SS_{A}$ for every $\varphi\in\Lang(A)$. 
\item A function symbol $\ff_{\alpha}:\SS_{A}\to \SS_{B}$ for every morphism $\alpha:A\to B$ in $\C$. 
\end{enumerate}
The interpretation $\M_\C:\Ell_\C\to\C$ assigns to every sort, relation and function symbol the object it came from. We define the theory $T_\C$ to be the set of all sentences $\varphi$ such that $\M_\C(\varphi) = 0$. With this definition, $\M_\C$ is a model $(\Ell_\C, T_\C)\to\C$. 
\end{definition}

One bit of sanity checking is the following:
\begin{proposition}\label{soundness}
Suppose that $\C$ is a continuous syntactic category and $T_\C \models \varphi$ for some sentence $\varphi$ in $\Ell_\C$ then $\varphi \in T_\C$.
\end{proposition}

\proof This is essentially the soundness theorem for continuous logic as worked out in \cite{a-proof-of-completeness}.  We give a brief sketch of the proof.  In \cite{a-proof-of-completeness}, a proof system is presented for continuous logic.  For a theory $T$, $T \models \varphi$ iff $T \proves \varphi \dotminus \epsilon$ for every $\epsilon > 0$; the proof's are carried out in the proof system that they present.  It suffices then to notice two things: first of all, all the axioms presented in \cite{a-proof-of-completeness} are true in the model $\M_\C$ and the only proof rule, modus ponens, preserves truth in $\M_\C$.  So if $T \models \varphi$ it will follow that $(\varphi \dotminus \epsilon) \in T_\C$ for all $\epsilon > 0$.  The second thing to notice is that 
$\Lang(1)$ is a standard continuous logical algebra.  The statement that $(\varphi \dotminus \epsilon) \in T_\C$ for all $\epsilon > 0$ means that in the continuous logical algebra $\Lang(1)$, $\varphi$ evaluates to 0 and is hence in $T_\C$.

It is somewhat tedious to check all the axioms presented in \cite{a-proof-of-completeness}; we look at them in batches.  The first set are the axioms A1--A6 which govern the propositional portion of the proof system.  These are all equations which hold in $[0,1]$ viewed as a continuous logical algebra and hence are true in $\M_\C$.  The axioms A10--A14 which govern the metric symbols are not involved in the proof system for $[0,1]$-valued logic but see Proposition \ref{metric-proof-system} after we introduce metric symbols in our language.
This leaves us with the axioms A7--A9 which govern the quantifiers; we check A7 and leave the other two to the interested reader. 

The axiom A7 from \cite{a-proof-of-completeness} is  
\[
(\forall_\var{x}  \psi \dotminus  \forall_\var{x} \varphi) \dotminus \forall_\var{x}(\psi \dotminus \varphi)
\]
Implicit in the notation is the possibility that $\psi$ and $\varphi$ are formulas in the free variables $\var{x}$ and $\tvar{y}$.  So if $Y$ is the object associated to the variables $\tvar{y}$, this axiom should be true in the continuous logical algebra $\Lang(\tvar{y})$.  Rearranging things in that algebra alone, we wish to prove that
\[
\forall_\var{x} \psi \leq \forall_\var{x}(\psi \dotminus \varphi) + \forall_\var{x} \varphi.
\]
But if $\Lang(\pi):\Lang(\tvar{y}) \to \Lang(\var{x}\tvar{y})$ is the embedding to which $\forall_\var{x}$ is an adjoint then we have $\theta \leq \Lang(\pi)\circ \forall_\var{x} \theta$ for any $\theta \in \Lang(\var{x}\tvar{y})$ and this guarantees that
\[
\forall_\var{x} \max\{\psi,\varphi\} = \forall_\var{x} ((\psi \dotminus \varphi) + \varphi) \leq \forall_\var{x}(\psi \dotminus \varphi) + \forall_\var{x} \varphi
\]
and since $\forall$ is order preserving the result follows. \qed

\begin{definition}\label{def:interpretation}
Let $\Ell$ and $\Ell'$ be continuous languages. An {\em interpretation} $I$ of $\Ell$ in $\Ell'$, which we write $I:\Ell\to\Ell'$ assigns
\begin{enumerate}
\item to every sort symbol $\SS$ of $\Ell$, a sort $I(\SS)$ of $\Ell'$
\item to every function symbol $\ff:\prod\SS_i\to\SS$, a term $I(\ff):\prod I(\SS_i)\to I(\SS)$, and 
\item to every relation symbol $\RR$ with domain $\SS_1\times\cdots\times\SS_n$ in $\Ell$ a formula $I(\RR)$ with domain $I(\SS_1)\times\cdots\times I(\SS_n)$ 
\end{enumerate}
\end{definition}

 This is clearly not a general enough definition of interpretation since sorts can only be interpreted as sorts. However, the definition simplifies the presentation and it is enough for this paper, since many of the languages we will consider will have sort symbols corresponding to every definable set. If $I:\Ell\to\Ell'$ is an interpretation, and $\sigma(\tvar{x})$ is a sentence, then $I(\sigma(\tvar{x}))$ is the natural interpretation of $\sigma$ via $I$. 
 
\begin{definition}
Let $(\Ell, T)$ and $(\Ell', T')$ be continuous theories. An {\em interpretation of  $(\Ell, T)$ in  $(\Ell', T')$} is an interpretation $I$ of $\Ell$ in $\Ell'$ such that whenever $\sigma$ is a sentence in $\Ell$, we have that  $T\models \sigma$ only if $T'\models I(\sigma)$. We will denote an interpretation of $(\Ell, T)$ in $(\Ell', T')$ by an arrow $I:(\Ell, T)\to(\Ell', T')$.
\end{definition}

If $\M:(\Ell', T')\to\R$ is a model and $I:(\Ell, T)\to (\Ell', T')$ is an interpretation then we may compose $I$ with $\M$ and get a model $\M\circ I:(\Ell, T)\to\R$. This defines a functor $I^*_\R:\Mod_\R(\Ell', T')\to\Mod_\R(\Ell, T)$, which we call the {\em forgetful functor associated to $I$}.

\begin{theorem}[The Universal Property of $(\Ell_\R, T_\R)$ with respect to $\R$]\label{thm:universal-property-of-LR}
Let $\R$ be a continuous syntactic category, and let $(\Ell_\R, T_\R)$ be the canonical language associated to $\R$. Let $\M_\R:(\Ell_\R, T_\R)\to\R$ be the canonical interpretation. Then for any theory $(\Ell, T)$, and any $\R$-model $\M:(\Ell, T)\to\R$, there is an interpretation $I:(\Ell, T)\to (\Ell_\R, T_\R)$ such that the following commutes:
\[
\xymatrix{{(\Ell, T)\rto^\M\dto_{I}}&{\R}\\{(\Ell_\R,T_\R)\urto_{\M_\R}}&{}}
\] 
\end{theorem}

\begin{proof}
By definition the arrow $\M_\R$ is a bijection between the sorts of $\Ell_\R$, and the objects of $\R$. Define $I$ by assigning to every sort symbol $\SS$ of $\Ell$ the sort $\M_0^{-1}(\M(\SS))$ of $\Ell_\R$. If $\ff$ is a function symbol of $\Ell$, then $\M(\ff)$ is a morphism of $\R$, which corresponds to a function symbol of $\Ell_\R$, and we can define $I(\ff)=\M_\R^{-1}(\M(\ff))$. If $\RR$ has domain $\SS_1\times\cdots\times\SS_n$ is a relation symbol of $\Ell$, then $\M(\RR)\in\Lang_\R(\M(\SS_1)\times\cdots\times\M(\SS_n))$, and we can define $I(\RR)=\M_\R^{-1}(\M(\RR))$.  It it clear from this definition that $\M=\M_\R\circ I$. 

Let $\sigma$ be a sentence of $\Ell$, and suppose $T\models \sigma$. Since $\M$ is a model of $T$, by we have $\M\models\sigma$, which implies that $\M_\R\models \M(\sigma)$, the translation of $\sigma$ into $\R$. By definition, $\M(\sigma)$ is an element of $T_\R$, so  $T_\R\models I(\sigma)$, showing that $I$ is indeed an interpretation. 
\end{proof}

\section{Metric Logical Categories}\label{sec:metric-logical-category}






\begin{definition} 
Let $S\in \C$, and $d\in\Lang(S\times S)$.  We will call $d$ a {\em pseudo-metric on $S$} if and only if $\M_\C$ satisfies all the sentences listed in Definition \ref{def:pseudo-metric}.
\end{definition}

\begin{proposition}\label{prop:equal-terms-implies-zero-distance}
Let $\R$ be a continuous syntactic category, and let $B\in\R$. Suppose $d$ is a pseudo-metric on $B$, and let $t(\tvar{x})$ and $s(\tvar{x})$ be two terms of $\Ell_\R$ with values in $B$. If $\M_\R(t(\tvar{x}))=\M(s(\tvar{x}))$, then $\M_\R\models d(t(\tvar{x}),s(\tvar{x}))$.
\end{proposition}

\begin{proof}
Let $f=\M(t(\tvar{x})):\M(\tvar{x})\to B$ and $g=M(s(\tvar{x})):\M(\tvar{x})\to B$. By definition, the interpretation of $d(t(\var{x}),s(\var{y}))$ as an element of $\Lang(\DOM(\tvar{x})\times \DOM(\tvar{x}))$ is given by $\Lang(F)(d)$, where $F:\DOM(\tvar{x})\times \DOM(\tvar{x})\to B\times B$ is the unique morphism satisfying $\pi_1^BF=f\pi_1^A$ and $\pi_2^BF=g\pi_2^A$. Since $d$ is a pseudo-metric, $\R\models d(\var{x},\var{x})$ which by definition means that $\Lang(\Delta_B)(d)=0$ for where $\Delta_B:B\to B\times B$ is the diagonal embedding. By the definition of $F$, and because of the assumption that $f=g$ as morphisms in $\R$, we have $F\Delta_A=\Delta_B f=\Delta_B g$. The interpretation of $d(f(\var{x}),f(\var{x}))$ as an object of $\Lang(A)$ is given by $\Lang(F\Delta_A)(d)=\Lang(f)(\Lang(\Delta_B)(d)$. This shows that $\R\models d(t(\var{x}),s(\var{x}))$. 
\end{proof}

\begin{definition}
If $d_A$ is a pseudo-metric on $A$, $d_B$ is a pseudo-metric on $B$ and $\alpha:A\to B$ is an arrow in $\C$, then we say that $\alpha$ is {\em uniformly continuous with respect to $d_A$ and $d_B$} if and only if for every $\varepsilon>0$, there exists $\delta>0$ such that 
\[
T_\C\models d_A(x,y)<\delta\Implies d_B(\alpha(x),\alpha(y))\leq\varepsilon .
\] 
If $R\in\Lang(A)$, then we say that $\R$ is {\em uniformly continuous with respect to $d_A$} if and only if for every $\varepsilon>0$, there exists $\delta>0$ such that 
\[
T_\C\models d_A(x,y)<\delta\Implies |R(x)-R(y)|\leq\varepsilon
\]
\end{definition}


\begin{definition}
A {\em metric logical category} is a pair $(\R, d)$, where $\R$ is a syntactic category, and $d$ is an assignment of a pseudo-metric $d_B\in\Lang(B\times B)$ to every object $B\in\R$, such that 
\begin{enumerate}
\item for every $A,B\in \R$, every morphism $\alpha:A\to B$ is uniformly continuous with respect to $d_A$ and $d_B$, and every element $b\in\Lang(B)$ is uniformly continuous with respect to $d_B$. 
\item for every $A,B\in\C$, and $\alpha,\beta\in\hom(A, B)$, if $T_\R\models d_A(\ff_\alpha,\ff_\beta)$ then $\alpha = \beta$
\end{enumerate}
\end{definition}

The metric symbol will be treated as a special symbol.  This is not entirely necessary but makes the presentation somewhat simpler.  So when interpreting a metric language $\Ell$ in a {\em metric} logical category, we will require that for every sort $\SS$, the metric symbol $\dd_\SS$ be interpreted by the metric $d_{\M(\SS)}$. The congruence properties for $d$ we are requiring of $T_\C$ will then force the interpretation of all function and relation symbols, as well as all terms and formulas, to be uniformly continuous with respect to the chosen metrics.

Let's see that $\Metric$ is a metric logical category.  If we assign to every object $X$, its metric $d_X$ then by definition of the category, all morphisms and all elements of $\Lang(X)$ are uniformly continuous.  Now suppose that $\alpha,\beta:X \to Y$ are two uniformly continuous functions between two metric spaces in $\Metric$, $X$ and $Y$.  If $\alpha \neq \beta$ then there is $a \in X$ such that $\alpha(a) \neq \beta(a)$. Then $d_X(\alpha(a),\beta(a)) \neq 0$ and so $d_X(\ff_\alpha,\ff_\beta)$ is not in $T_\Metric$. This shows that $\Metric$ is a metric logical category.

The other important example of a metric logical category is $\DEF(\Ell,T)$ for a metric theory
 $(\Ell, T)$ By Lemma 1.10 in \cite{definability-of-groups}, $\DEF(\Ell, T)$ is closed under the formation of all finite products.
 Now suppose we have a $T$-definable set $A(\tvar{x})$; we wish to define the continuous logical algebra $\Lang(A(\tvar{x}))$.
 For two formulas $\varphi(\tvar{x})$ and $\psi(\tvar{x})$ we say $\varphi \sim_A \psi$ if $T\models \forall_{A(\tvar{x}) = 0} |\varphi(\tvar{x}) - \psi(\tvar{x})|$. 
  Let $\Lang(A(\tvar{x}) )$ be the equivalence classes of $\sim_A$.  
  The definition of the continuous logical operations is almost tautological: if $\uu:[0,1]^n \to [0,1]$ is a continuous function and 
  $\varphi_1(\tvar{x}),\ldots,\varphi_n(\tvar{x})$ are formulas then
 \[
 \uu([\varphi_1],\ldots,[\varphi_n]) = [\uu(\varphi_1,\ldots,\varphi_n)]
 \]
 where $[\psi]$ is the $\sim_A$-class of $\psi$.  Notice that $\Lang(1)$ will simply be the $T$-equivalence classes of sentences.  If $T \models \varphi \leq \varepsilon$ for all $\varepsilon > 0$ then $T \models \varphi$.  This says that $\Lang(1)$ is a standard continuous logical algebra.
  
Now if $\alpha(\tvar{x},\tvar{y})$ is a definable function from $A(\tvar{x})$ to $B(\tvar{y})$ then we can define a natural map from $\Lang(B)$ to $\Lang(A)$ by composing with $\alpha$.  This yields a contravariant functor from $\DEF(\Ell,T)$ to $\CLA$.   If $A$ and $B$ are two definable sets, let $\pi:A \times B \to B$ be the definable function which projects onto $B$.  $\Lang(\pi)$, composition with $\pi$ is clearly a continuous logical embedding from $\Lang(B)$ to $\Lang(A \times B)$.  Of course it has a left adjoint given by 
\[
\forall_\pi (\varphi(\tvar{x},\tvar{y})) := \forall_{A(\tvar{x}) = 0} \varphi(\tvar{x},\tvar{y}).
\]
This shows that $\DEF(\Ell,T)$ is a continuous syntactic category.  The uniform continuity requirement follows from the fact that $T$ knows that each formula is uniformly continuous with respect to the metrics on each sort.  Finally, if $\alpha$ and $\beta$ are two definable functions from $A$ to $B$ and $T\models d_A(\alpha,\beta)$ then $\alpha$ and $\beta$ are $T$-equivalent and hence equal in $\DEF(\Ell,T)$.

\begin{definition}
Let $(\Ell, T)$ be a metric theory. We define the {\em canonical interpretation} $\M_{(\Ell, T)}:(\Ell, T)\to\DEF(\Ell, T)$ as follows. To the sort $\SS$ we assign the object $[\dd(\var{x}, \var{x})]$, where $\var{x}$ is a variable of sort $\SS$. To the function symbol $\ff:\SS_1\times\cdots\times\SS_n\to\SS$ we assign the functional formula $\alpha_\ff(\tvar{x}, \tvar{y})=\dd_\var{y}(\ff(\tvar{x}), \var{y})$, and to the relation symbol $\RR\subseteq\SS_1\times\cdots\times\SS_n$ we assign the formula $\RR(\var{x}_1,...,\var{x}_n)\in\Lang(\var{x}_1\var{x}_2\cdots\var{x}_n)$.
\end{definition}

\begin{proposition}\label{metric-proof-system}
Suppose that $\R$ is a metric logical category and $T_\R \models \varphi$ for some sentence $\varphi$.  Then $\varphi \in T_\R$.
\end{proposition}

\begin{proof} This is soundness for metric logical categories and is an analogue of Proposition \ref{soundness}.  We again rely on the proof system presented in \cite{a-proof-of-completeness}.  As in Proposition \ref{soundness}, one needs only check that the axioms and proof rules presented in that paper preserve validity in our setting.  The only axioms we did not discuss in our earlier proof were axioms A10--14.  One observes that these are all equational axioms and they hold in the standard continuous logical algebra $[0,1]$.  Hence they hold in any continuous logical algebra and we are done.
\end{proof}

\begin{definition}
A {\em logical functor} from a metric logical category $(\R,\Lang_\R)$ to another $(\S,\Lang_\S)$ is a pair $I$ and $i$ such that
\begin{enumerate}
\item  $I:\R\to\S$ is a product preserving functor,
\item  $i$ is a natural transformation between $\Lang_\R$ and $\Lang_S \circ I$,
\item for every $A \in \R$, $i$ induces a continuous logical algebra embedding from $\Lang_\R(A)$ to $\Lang_\S(I(A))$, and
 \item for every $S\in\R$, $\iii_S(d^\R_S)=d^\S_{\FF(S)}$.
 \end{enumerate}
We will denote by $\hom(\R, \S)$ the class of all logical functors from $\R$ to $\S$.
\end{definition}

The final condition here is again technically not necessary but is in line with our assumption regarding interpretations of metric symbols.

\begin{theorem}[The Universal Property of $\R$ with respect to $(\Ell_\R, T_\R)$]\label{thm:universal-property-of-def-1}
Let $\R$  and $\S$ be metric logical categories, and let $(\Ell_\R, T_\R)$ be the canonical language associated to $\R$. 
For any model $\M:(\Ell_\R, T_\R)\to\S$, there is a logical functor $I:\R\to \S$ such that $$\xymatrix{{(\Ell_\R, T_\R)\rto^{\M_\R}\drto_{\M}}&{\R\dto^I}\\{}&{\S}}$$ is commutative.
\end{theorem}

\begin{proof}
By definition, $\M_\R$ defines a bijection between the sorts of $(\Ell_\R, T_\R)$ and the objects of $\R$. If $A\in\R$, then we define $I(A)=\M(\SS_A)=\M(\M_\R^{-1}(A))$. If $\alpha:A\to B$ is a morphism, then we have $\ff_\alpha:\SS_A\to\SS_B$ as a function symbol in $\Ell_\R$. Define $I(\alpha)=\M(\ff_\alpha)=\M(\M_\R^{-1}(\alpha))$. If $\varphi\in\Lang_\R(A)$, then in $\Ell_\R$ we have the predicate symbol $\RR_\varphi$, and we can define $I(\varphi)=\M(\RR_\varphi)=\M(\M_\R^{-1}(\varphi))$.

First we show that $I$ thus defined is indeed a functor. Suppose $\alpha:A\to B$ and $\beta:B\to C$ are morphisms in $\R$, and $D$ is any other object. Consider $\beta\alpha$. We have $\beta \circ \alpha = \beta\alpha$ and so we have 
\[
T_\R\models d_A(\ff_{\beta\alpha}, \ff_\beta(\ff_\alpha)). 
\]
 Since we are assuming that $\M$ is a model of $T_\R$, $\M$ satisfies $d_A(\ff_{\beta\alpha}, \ff_\beta(\ff_\alpha)) = 0$. This shows that in $\S$, $d_{I(A)}(\ff_{I(\beta\circ\alpha)}, \ff_{I(\beta)} \circ \ff_{I(\alpha)})$ evaluates to 0 and since $\S$ is a metric logical category, we have
$I(\beta\circ\alpha)=I(\beta)\circ I(\alpha)$
showing that $I$ is a functor.


We now need to define $i$.  Suppose $A\in\R$ and $\varphi \in \Lang(A)$.  Then $R_\varphi \in \Ell_\R$ and we can define 
$i_A:\Lang(A) \to \Lang(I(A))$ by $i_A(\varphi) = \M(\RR_\varphi)$.  We need to see that this defines a continuous logical algebra embedding and that it is natural between $\Lang_\R$ and $\Lang_\S \circ I$. Towards the first, 
suppose $u:[0,1]^n\to[0,1]$ is continuous, and $b_1,...,b_n\in\Lang(A)$. Let $b=\ff_u(b_1,..,b_n)$, and consider the formula $|\RR_b - \uu(\RR_{b_1},...,\RR_{b_n})|$. Since $b=\ff_u(b_1,..,b_n)$, this formula is in $T_\R$ and so $\M$ satisfies it in $\S$.  Unravelling this, this means that $\M(\RR_b) = \uu(\M(\RR_{b_1}),\ldots,\M(\RR_{b_n}))$ which shows that $i_A$ is a continuous logical algebra homomorphism.  That it is an embedding follows from the fact that if $\M$ satisfies $|\RR_a - \RR_b|$ then this must be true in $T_\R$ and hence it must be the case that $a = b$.

%
%
%
Now to show that $i$ is natural, we need to suppose that $f:A \to B$ in $\R$ and show that $i_A\circ \Lang(f) = \Lang(I(f))\circ i_B$.  But this is known to the theory for if $b \in \Lang(B)$ then $\Ell_\R$ has an associated $\RR_b$.  If $a = \ff(b)$ then $T_\R$ satisfies $|\RR_b(\ff) - \RR_a|$.  Since $\M$ is a model of $T_\R$ in $\S$ we see that $i_A(a) = I(\ff)(i_B(b))$ which is the commutativity that we need.
\end{proof}
\begin{corollary}\label{thm:logical-functor-interpretation}
Let $\C$ and $\D$ be metric logical categories. There is a one-to-one correspondence between logical functors $\C\to\D$ and interpretations $(\Ell_\C,T_\C)\to(\Ell_\D,T_\D)$.
\end{corollary}

\proof Suppose that $I:\C \to \D$ is a logical functor.  We have the canonical models $M_\C$ and $M_\D$ of the theories $(\Ell_\C,T_\C)$ and $(\Ell_\D,T_\D)$ respectively.  Then $I \circ \M_\C$ is a $\D$-model of $(\Ell_\C,T_\C)$.  By Theorem \ref{thm:universal-property-of-LR}, this induces an interpretation from $(\Ell_\C,T_\C)\to(\Ell_\D,T_\D)$.  On the other hand, if we have an interpretation $I$ from $(\Ell_\C,T_\C)\to(\Ell_\D,T_\D)$ then $\M_\D \circ I$ is a $\D$-model of $(\Ell_\C,T_\C)$.  By Theorem \ref{thm:universal-property-of-def-1}, this induces a logical functor from $\C\to\D$.  These processes are clearly inverses of each other which is what we wished to show. \qed

We now wish to consider the collection of metric logical functors between two metric logical categories as a category itself.  Towards this end we define
\begin{definition}
Let $\FF,\GG:\R\to\S$ be logical functors, and write $\FF=(\FFF, \iii_\FF)$ and $\GG=(\GGG, \iii_\GG)$. A {\em logical transformation} $h:\FF\to\GG$ is a pair $h=(\eta, \epsilon)$ where $\eta:\FFF\to\GGG$ and  $\epsilon:\Lang\GGG\to\Lang\FFF$ are a natural transformations such that for every $A\in\R$, the diagram 
$$\xymatrix{
{\Lang(A)\ar[r]^{\iii_{\FF,A}}\ar[dr]_{\iii_{\GG,A}}}&{\Lang(\FFF(A))}\\
{}&{\Lang(\GGG(A))\ar[u]_{\epsilon_A}}
}$$ is commutative.
\end{definition}
The category of metric logical functors between $\R$ and $\S$ with logical transformations as morphisms will be denoted $\hom^*(\R,\S)$. 

\begin{theorem}\label{thm:models-are-equivalent-to-logical-functors}\label{thm:models-and-logical-functors-on-C}
Let $\C$ be a metric logical category. Then there is an equivalence of categories $$\Mod^*(\Ell_\C,T_\C)\cong\hom(\C,\Metric)$$
\end{theorem}

\begin{proof}
Let $\M,\N:(\Ell_\C, T_\C)\to\Metric$ be models, and suppose $h:\M\to\N$ is an elementary map. Let $A$ be an object in $\C$, and let $\SS_A$ be the corresponding sort in $\Ell_\C$. From the definition of $F_\M$ and $F_\N$, we have that $F_\M(A)=\M(\SS_A)$, and $F_\N(A)=\N(\SS_A)$. Define $\eta_A:F_\M(A)\to F_\N(A)$ via $\eta_A(x)=h_{\SS_A}(x)$. 

Let $\alpha:A\to B$ be a morphism in $\C$, and let $\ff_\alpha:\SS_A\to\SS_B$ be the corresponding function symbol in $\Ell_\C$. Then by definition $$\eta_B F_\M(\alpha)=h_{\SS_B}\M(\ff_\alpha)=\N(\ff_\alpha) h_{\SS_A}=F_\N(\alpha)\eta_A$$ the equality in the middle holds because $h$ is elementary. For the definition of $\varepsilon:\Lang F_\N\to\Lang F_\M$, suppose $b\in\Lang(F_\N(A))$. By definition, $b$ is a function $\N(\SS_A)\to[0,1]$. Since $h_{\SS_A}:\M(\SS_A)\to\N(\SS_B)$, we get that $b h_{\SS_A}:\M(\SS_A)\to [0,1]$. We define $\varepsilon_A(b)= b h_{\SS_A}$.

If $\alpha:A\to B$ is a morphism in $\C$ and consider the diagram
$$\xymatrix{
{\Lang(F_\N(A))\rrto^{\varepsilon_A}}&{}&{\Lang(F_\M(A))}\\
{\Lang(F_\N(B))\rrto^{\varepsilon_B}\ar[u]^{\Lang(F_\N(\alpha))}}&{}&{\Lang(F_\M(B))\ar[u]_{\Lang(F_\M(\alpha))}}
}$$ By definition, $\Lang(F_\N(\alpha))(b)=b\circ F_\N(\alpha)$, so that $$\varepsilon_A(\Lang(F_\N(\alpha))(b))=\epsilon_A(b\circ F_\N(\alpha))=b\circ F_\N(\alpha)\circ h_S$$ and $\Lang(F_\M(\alpha))(g)=g\circ F_\M(\alpha)$, so that $$\Lang(F_\M(\alpha))(\epsilon_B(b))=\Lang(F_\M(\alpha))(b\circ h_R)=b\circ h_R\circ F_\M(\alpha)$$ Since $h$ is an elementary map, $h_R\circ F_\M(\alpha)=F_\N(\alpha)\circ h_S$, which shows that the diagram is commutative, and that $\epsilon$ is a natural transformation.

Finally, for every object $A\in\C$, if $\varphi\in\Ell(A)$, then $\iii_{\M,A}(\varphi)=\M(\varphi)$ by definition, so that $\iii_{\M,A}(\varphi)=\M(\varphi)=\N(\varphi)\circ h_S=\iii_{\N,A}(\varphi)\circ h=\epsilon_A(\iii_{\N,A})$, which is the required property for $\epsilon$. This completes the proof that $(\eta, \epsilon)$ is a logical transformation. It is clear from the definition that this definition commutes with composition.

Now let $F,G:\C\to\Metric$ be logical functors, and consider the models $\M_F$ and $\M_G$. In what follows if $\SS$ is a sort of $\Ell_\C$, we write $A_\SS$ for the corresponding object of $\C$. Define $h:\M_F\to\M_G$ via $h_\SS(x)=\eta_{A_\SS}(x)$. Since $\eta$ is a natural transformation, it is clear that $h_\SS$ commutes with all the function symbols of $\Ell_\C$. 

It remains to show that $h$ thus defined preserves the value of all formulas of $\Ell_\C$. Let $\varphi(\tvar{x})$ be a formula of $\Ell_\C$. Then
\begin{eqnarray*}
\N(\varphi(\tvar{x}))(h(\overline{x}))&=& \N(\varphi(\tvar{x}))(h_{\SS_1}(x_1), ..., h_{\SS_n}(x_n))\\
{} &=& \iii(\varphi(\tvar{x}))(\eta_{A_{\SS_1}}(x_1), ..., \eta_{A_{\SS_n}}(x_n))\\
{} &=&\epsilon\iii(\varphi(\tvar{x}))(x_1, ..., x_n)
\end{eqnarray*}

\end{proof}

\begin{corollary}\label{cor:satisfaction-via-metric-models}
Let $\R$ be a metric logical category, then  for every statement $\sigma(\tvar{x})$, we have $\R\models\sigma(\tvar{x})$ if and only if $\M\models\sigma(\tvar{x})$ for every metric logical functor $\M:\R\to\Metric$.
\end{corollary}

\section{Conceptual completeness: model theory version}\label{sec:completion-of-languages}

\subsection*{Imaginaries in continuous logic}
Before we give the proof of the conceptual completeness theorem, we need to remind the reader of the construction of imaginaries in continuous logic.  There are a couple extra wrinkles beyond the discrete first order case.  Presentations of imaginaries in continuous logic appear in \cite{scc-continuous-logic}, \cite{model-theory-for-metric-structures}, \cite{hart-makkaifest} and \cite{muenster-paper}.

Suppose that $(\Ell,T)$ is a metric theory.  Then $(\Ell^{\eq},T^{\eq})$ is the smallest expansion of $(\Ell,T)$ satisfying the following closure properties:
\begin{example}\label{Ex.axioms}\mbox{ }
\begin{enumerate}
\item \label{Ex.ctble-prod}{\bf Closure under countable products:} If $(\SS_n : n < \omega)$ is a sequence of sorts in $\Ell^\eq$ then there is a sort $\SS$ in $\Ell^\eq$ with metric symbol $d_\SS$ together with function symbols
$\pi_n:\SS \rightarrow \SS_n$.  $T^\eq$ contains, for all $n < \omega$, the sentences
\[
\forall_{x_1\in \SS_1}\ldots\forall_{x_n\in \SS_n} \exists_{y\in \SS} \bigwedge^n_{i=1} d_i(\pi_i(y),x_i)
\]
where $d_i$ is the metric on $\SS_i$, and
\[
\forall_{x,y \in \SS} | d_\SS(x,y) - \sum_{i=1}^n \frac{d_i(\pi_i(x),\pi_i(y))}{2^i} | \dotminus \frac{1}{2^n}
\]
\item\label{Ex.def-sets}{\bf Closure under definable sets:} If $A(\var{x}_1,...,\var{x}_n)$ is a definable set in $T^\eq$
then there is a sort $\SS_A$ in $\Ell^\eq$ with metric symbol $\dd_A$, and function symbols $\ff_i:\SS_A\to\DOM(\tvar{x}_i)$ for $1\leq i\leq n$.
$T^\eq$ contains the sentences
\[
|A(\var{x}_1,...,\var{x}_n) - \exists_{\var{y}}\max\{\dd_1(\var{x}_1,\ff_1(\var{y})),ldots,\dd_n(\var{x}_n,\ff_n(\var{y}))\}|
\]
where $\dd_i$ is the metric symbol on $\SS_i$, and
\[
|\dd_A(\tvar{x}, \tvar{y})-\dd_1(\ff_1(\tvar{x}),\ff_1(\tvar{y}))\And \cdots\And \dd_n(\ff_n(\tvar{x}),\ff_n(\tvar{y}))|
\]

\item\label{Ex.can-para}{\bf Closure under canonical parameters:} If $\varphi(\tvar{x},\tvar{y})$ is a formula in $\Ell^\eq$ then there is a sort $\SS_\varphi$ in $\Ell^\eq$ with metric symbol $\dd_\varphi$ and a function symbol $\pi_\varphi:\DOM(\tvar{y}) \rightarrow \SS_\varphi$.  $T^\eq$ contains the following sentences:
\[
\forall_{y,y'} |\dd_\varphi(\pi_\varphi(y),\pi_\varphi(y')) - \forall_x(\varphi(x,y) - \varphi(x,y'))|
\]
and 
\[
\forall_z \exists_y (\dd_\varphi(\pi_\varphi(y),z)
\]
\item\label{Ex.unions}{\bf Closure under finite unions:} If $\varphi_1(\tvar{x},\tvar{y_1}),\ldots,\varphi(\tvar{x},\tvar{y_n})$ are formulas in $\Ell^\eq$ then there is a sort $\SS$ in $\Ell^\eq$ with metric symbol $\dd$, and function symbols $i_j:\SS_{\varphi_j} \rightarrow \SS$ in $\Ell^\eq$.  $T^\eq$ contains the sentences
\[
\forall_{x \in S} \bigvee_{j=1}^n \exists_y \dd(x,i_j(y))
\]
and for all $1\leq j,k \leq n$
\[
\forall_y\forall_z |\dd(i_j(y),i_k(z))-\forall_x|\varphi_j(x,y) - \varphi_k(x,z)||
\]

\end{enumerate}
\end{example}

A few comments are in order.  
\begin{enumerate}
\item Regarding closure under countable products, it follows from the two axioms listed that the sort $\SS$ is bijective with $\prod_n \SS_n$ and that the metric on $\SS$ is induced by the metric $\sum_i \frac{d_i}{2^i}$.  
\item For closure under definable sets, the theory is expanded so that quantification over a definable set is provided by quantification over its own sort.  
\item For a formula $\varphi(\tvar{x},\tvar{y})$, the function induced on $\DOM(\tvar{x})$ by $\varphi$ is captured by the image of $\tvar{y}$ in the sort $\SS_\varphi$ which mimics the construction of canonical parameters in the discrete case. We abuse notation by writing $\varphi(\tvar{x},\var{y})$ for $\var{y}$ of sort $\SS_\varphi$ to mean $\varphi(\tvar{x},\tvar{y})$ for any $\tvar{y}$ such that $\pi_\varphi(\tvar{y}) = \var{y}$. 
\item It is convenient although not entirely necessary to consider finitely many formulas all inducing functions on the same sort and to take the union of the corresponding sets of canonical parameters.  We include the finite union of such canonical parameter sorts to cover the most general situation.  In practice this can usually be avoided.  For instance, suppose one has two sorts $\SS_1$ and $\SS_2$ as described in 4, i.e.\ functions from a single sort $X$ to $[0,1]$, and a single sort $C$ with exactly two elements designated 0 and 1.  We can then form $C \times \SS_1 \times \SS_2$ together with the function on $X$, $\varphi(i,a,b)$ defined by:
\[
\varphi(0,a,b)(x) = a(x) \text{ and } \varphi(1,a,b)(x) = b(x).
\]
If we quotient by the kernel of $\varphi$, the resulting object is effectively the union of $\SS_1$ and $\SS_2$.  Of course this can be repeated for any finite number of  sorts as in 4 and so the closure under finite unions is covered by the other three clauses whenever there is a sort with two distinct constants.
\end{enumerate}
It is reasonably clear that each of these closure properties provides a conservative expansion to the theory $T$ and so we record the following

\begin{theorem}\label{Teq-is-conservative}
If $(\Ell,T)$ is a metric theory then the forgetful function 
\[
F:Mod(\Ell^\eq,T^\eq) \rightarrow Mod(\Ell,T)
\]
 is an equivalence of categories i.e. $T^\eq$ is a conservative expansion of $T$.
\end{theorem}


\begin{definition}
Suppose $(\Ell, T)\subseteq (\Ell', T')$ are metric theories and $\M'\in\Mod(\Ell',T')$. Let $\M$ be its $\Ell$-reduct.  We say that $\M$ is {\em stably embedded in $\M'$} if and only if for every $\varepsilon>0$, and every $\Ell'$-formula $\varphi(\tvar{x}, \tvar{y})$, where $\DOM(\tvar{x})\in\Ell$ and $\DOM(\tvar{y})\in\Ell'$, there is an $\Ell$-formula $\psi(\tvar{x}, \tvar{z})$ such that  for every $\tvar{a} \in \DOM(\tvar{y})$ there is $\tvar{b} \in \DOM(\tvar{z})$
\[
\M'\models \forall_\tvar{x}|\varphi(\tvar{x}, \tvar{a})-\psi(\tvar{x},\tvar{b})|\leq\varepsilon.
\]
\end{definition}  

The above definition can be transferred easily to the case where we have an interpretation $I:(\Ell, T)\to (\Ell', T')$ by considering the image of $(\Ell, T)$ under $I$ as a subset of $\Ell'$. In this case we will say that $\M$ is stably embedded in $\M'$ via $I$.  

\begin{theorem}\label{lem:full-forgetful-functor-implies-stably-embedded}
	Let $I:(\Ell, T)\to(\Ell', T')$ be an interpretation, and consider the corresponding forgetful functor $F:\Mod(\Ell', T')\to\Mod(\Ell, T)$. If $F$ is full and faithful, then for every $\M'\in\Mod(\Ell', T')$, $F(\M')$ is stably embedded in $\M'$ via $I$.
\end{theorem}

\begin{proof}
Fix a model $\M'\in\Mod(\Ell', T')$ and let $\M = F(\M')$.  Suppose $\psi(\var{x}, \var{z})$ is a formula where $\var{x}$ is a variable of sort $\SS$ in $\Ell$ and $\var{z}$ is a variable of sort $\SS'$ in $\Ell'$. Fix $c \in \SS'(\M')$ and define the set $\Sigma(x,y,c;n,\psi)$ of statements in $\Ell'_{\M'}$ where $y$ is also of sort $\SS$ as
\begin{enumerate}
\item the elementary diagram of $\M'$ in $\Ell'$,	
\item for every $\Ell_{\M}$-formula $\varphi$ and $k\in\Nat$, $|\varphi(\var{x})-\varphi(\var{y})|\leq 1/k$, and
\item $|\psi(\var{x}, \var{c})-\psi(\var{y}, \var{c})|\geq 1/n$.
\end{enumerate}
Suppose $\Sigma(x,y,c;n,\psi)$ is consistent, and let $\N'\models\Sigma(a,b,c;n,\psi)$, with $a,b\in\SS(\N')$. Let $\N = F(\N')$. Note that there is an elementary embedding $g:\M'\to\N'$. Since $\Sigma(a,b,c;n,\psi)$ implies that $a\equiv_{\M} b$, there is an ultrafilter pair $(I,U)$ and an embedding $h:\N\to \N^U$ such that $h(a)=\Delta(b)$ and $\restr{h}{\M}=\Delta_{\M}$ where $\Delta$ is the diagonal embedding of $\N$ into $\N^U$. Since $F$ is full and faithful, there is a unique elementary map $h':\N'\to\N'^U$ such that $F(h')=h$ and $\restr{h'}{\N}=\Delta$.  Since $h'(a)=\Delta(b)$, we have $\psi(a,c)=\psi(b,c)$, 
which is impossible, since $a,b$ realize $\Sigma$ and $|\psi(a,c)-\psi(b,c)|\geq 1/n$. Therefore, $\Sigma(x,y,c;n,\psi)$ is inconsistent for every $n$ and every $\psi$.

 By compactness, for every $\varepsilon>0$, there is a number $\delta>0$ and a finite subset $\Delta_\varepsilon$ of $\Ell$-formulas such that if $\displaystyle\bigwedge_{\varphi\in\Delta}|\varphi(x)-\varphi(y)|<\delta$, then $|\psi(x,c)-\psi(y,c)|<\varepsilon$.   This says that $\psi(x,c)$ defines a continuous function on the set of types over the parameters in $\M$ in the variable $x$.  From this we conclude that $\psi(x,c)$ is equivalent to a definable predicate in the language $\Ell_\M$.
\end{proof}

\begin{theorem}[Conceptual Completeness]\label{thm:weak-conceptual-completeness}
Let $I:(\Ell, T)\to(\Ell', T')$ be an interpretation, and suppose that the forgetful functor $I^*:\Mod(\Ell', T')\to\Mod(\Ell, T)$ is an equivalence of categories. Then there is an interpretation $J:(\Ell',T')\to (\Ell^\eq, T^\eq)$ such that $JI$ is the inclusion $(\Ell, T)\subseteq(\Ell^\eq, T^{\eq})$
\end{theorem}

\begin{proof}
The strategy of the proof will be to take any sort $\SS$ in $\Ell'$ and show that there is a $T'$-definable injection $f:\SS \rightarrow \SS^*$ where $\SS^*$ is a sort in $T^{\eq}$.  Let's see why this will be enough.  The statement of conceptual completeness is a generalization of the classical theorem which we state in its continuous form.
\begin{theorem}[Beth's Theorem for Continuous Logic]\label{lem:beth-s-theorem}
Suppose that $(\Ell,T)$ is a metric theory, $T \subseteq T'$ and $T'$ is a metric theory in a language $\Ell'$ with no new sorts.  Further suppose that the forgetful functor $F:\Mod(\Ell',T') \rightarrow \Mod(\Ell,T)$ is an equivalence of categories.  Then every $\Ell'$-formula is $T'$-equivalent to a definable predicate in $\Ell$.
\end{theorem}
The image of $f$, $X$, will be a $T'$-definable subset of $\SS^*$ and since $T'$ is a conservative extension of $T$, by Beth definability, $X$ is also $T$-definable.  So since $T^{\eq}$ has a sort representing $X$ as a separate sort, say $\SS_X$, we can interpret $\SS$ as $\SS_X$. Let  $j := f^{-1}\circ i_X : \SS_X \rightarrow \SS$ be the map which definably provides a bijection between these two sorts and where $i_X$ is the embedding of $\SS_X$ into $\SS^*$.  Once we have identified sorts in $\Ell'$ with sorts in $T^{\eq}$ we still need to interpret function and relation symbols.  However this will come automatically again from the Beth definability theorem.  To see how let's assume for simplicity that we have a unary relation $R(x)$ where the sort of $x$ is $\SS$.  Consider the $\Ell'$-formula $R(j(y))$ where $y$ is a variable of sort $\SS_X$.  Again, since $T'$ is a conservative expansion of $T$, this formula is equivalent to a $T$-definable predicate on $\SS_X$.  Handling the cases of higher arity and function symbols requires some bookkeeping but is a similar proof.  So it will suffice to find the $f$ mentioned above.

Toward this end, let $\SS$ be a sort of $\Ell'$
and consider a formula $\varphi(\tvar{x}, \var{y})$, where $\DOM(\var{y})=\SS$, and 
$\DOM(\tvar{x})=I(\SS_1)\times\cdots\times I(\SS_k)$, where each $\SS_i$ is a sort of $\Ell$. 
By compactness and stable embeddedness, there is, for every $n$, a finite set $\Psi_n(\tvar{x}, \tvar{y}_1,...,\tvar{y}_{\ell(n)})=\{\psi_1(\tvar{x}, \tvar{y}_1),...,\psi_{\ell(n)}(\tvar{x}, \tvar{y}_{\ell(n)})\}$ of $\Ell$-formulas with the property that 
\[
T'\models \bigvee_{i=1}^{\ell(n)}\exists_{\tvar{y}_i}\big[|\varphi(\tvar{x}, y)-I(\psi_i(\tvar{x}, \tvar{y}_i))|\big]\leq\frac{1}{2^n}.
\]
 In $\Ell^\eq$, 
 let $\SS_i$ be the sort of canonical parameter of $\psi_i(\tvar{x}, \tvar{y}_i)$, and $\UU_n=\bigcup_{i=1}^{\ell(n)} \SS_i$; let's suppose that $d_n$ is the metric on $U_n$.
 Define $\SS^*_\varphi:=\prod_{n\geq 1}\UU_n$, and note that $\SS^*_\varphi$ is also a sort of $\Ell^\eq$.   
 
We need one small technical tool which is Lemma 3.7 from \cite{continuous-first-order-logic}.

\begin{lemma}
There is a continuous function $\text{Flim}: [0,1]^\omega \rightarrow [0,1]$ such that 
\begin{enumerate}
\item if $(a_n : n \in \omega)$ is a sequence in $[0,1]^\omega$ such that for all $m$, $|a_n - a_{n+1}| \leq 2^{-m}$ for all $n \geq m$ then
$\text{Flim}(a_n : n \in \omega) = \lim_n a_n$, and moreover
\item if $\lim_n a_n = b$ and $|b - a_n| \leq 2^{-n}$ then $\text{Flim}(a_n: n \in \omega) = b$.
\end{enumerate}
\end{lemma}

Consider the definable predicate $\psi(\bar a,\bar x) \Def \text{Flim}(a_n(\bar x))$ where we think of each element of $\UU_n$ as a function on $\DOM(\tvar{x})$.
 The pseudo-metric which captures the canonical parameters for $\psi$ is
 \[
 \rho(\overline{a}, \overline{b}) = \forall_{\tvar{x}}(\text{Flim}(a_n(\tvar{x}): n < \omega) - \text{Flim}(b_n(\tvar{x}): n < \omega)).
 \]
 Let  $\widehat{\SS_\varphi^*}=\SS_\varphi^*/\rho$ which is a sort in $T^\eq$.
 The point of this construction is that if we consider any model $\M$ of $T'$ and $c \in \SS(\M)$ we can choose $a_n \in \UU_n(\M)$ such that 
  \[
  \M \models |a_n(\tvar{x}) - \varphi(\tvar{x}, c)| \leq 2^{-n}.
  \]
    Of course the choice of $a_n$ is not unique and so this does not define a function from $\SS$ to $\SS^*_\varphi$. However any two sequences obtained this way are $\rho$-equivalent and hence this does define a function $f_\varphi:\SS \rightarrow \widehat{\SS_\varphi^*}$ with the property that $\varphi(x,c) = \varphi(x,c')$ iff $f_\varphi(c) = f_\varphi(c')$.

 
 Now consider $\Sigma_n(\var{c}, \var{c}')$ be the theory $T'$ together with the following set of formulas in the variables $\var{c}$ and $\var{c'}$ both of sort $\SS$:
\begin{enumerate}
\item All statements of the form $\forall_\var{x}|\varphi(\var{x},\var{c})-\varphi(\var{x},\var{c}')|\leq\varepsilon$, where $\DOM(\var{x})$ is a sort of $\Ell$, and $\varepsilon>0$
\item The statement $\dd(\var{c}, \var{c}')\geq 1/n$
\end{enumerate}
If $\Sigma_n(\var{c},\var{c}')$ is consistent then let $\N$ be a model of $\Sigma_n$, $\M$ the reduct to $\Ell$ and $c,c' \in \SS(\N)$ be witnesses for the two variables in $\Sigma_n$.  By assumption, the $\Ell'$-type of $c$ over $\M$ is the same as that of $c'$.  Now for any ultrafilter $\U$, let $\Delta:\N \rightarrow \N^\U$ be the diagonal embedding. By choosing a suitable ultrafilter $\U$, we can find an elementary map $h:\N \rightarrow \N^\U$ such that $h\restriction_\M = \Delta\restriction_\M$ and $h(c') = \Delta(c)$.  But we are assuming that $I^*$ is faithful which would imply that $h = \Delta$ contradicting that $\dd(c,c') \geq 1/n$.
We conclude then that $\Sigma_n(\var{c},\var{c}')$ is inconsistent for every $n$.  This implies that there is a countable set of $\Ell'$-formulas $\{\varphi_i(\var{x},\var{y}):i\in\Nat\}$ such that  if 
\[
T' \models \varphi_i(x,c)=\varphi_i(x,c')
\]
 for every $i<\omega$, then $c=c'$. Now let $\SS^*=\prod_{i\in\Nat}\widehat{\SS^*_{\varphi_i}}$, which is also a sort of $\Ell^\eq$. The sequence $f(c) = (f_{\varphi_i}(c):i<\omega)$ is an element of $\SS^*$, and $f(c)=f(c')$ implies that $c=c'$ by the previous argument.  Therefore, $f$ is the desired definable map from $\SS$ into $\SS^*$ and we are done.  
\end{proof}

\section{Pretoposes and the completion of a Metric Logical Category}

The concept of a pre-topos is introduced in Expos\'e VI of \cite{sga4}. A category $\C$ is a {\em pre-topos} if all finite projective limits are representable in $\C$, $\C$ has all finite sums, and the sums are disjoint, equivalence relations in $\C$ are effective, and all epi-morphisms in $\C$ are effective universal.  In \cite{first-order-categorical-logic}, it is shown that this is equivalent to saying that $\C$ is a logical category which is closed under the formation of quotients by equivalence relations, and the formation of finite disjoint sums. Furthermore, it is also shown that for every logical category $\C$, there is a pre-topos $\P(\C)$ and a conservative logical functor $I:\C\to\P(\C)$ which is universal among all conservative expansions of $\C$.  A direct corollary of the pre-topos completion in \cite{first-order-categorical-logic} is that a (boolean) pre-topos, when viewed as a logical category (and thus as a first-order theory), corresponds to a theory which eliminates imaginaries. The construction of the pre-topos completion of $\C$ in the classical framework is parallel to the construction of $(\Ell^\eq, T^\eq)$.

Grothendieck's notion of pre-topos is much too strong for the needs of continuous logic. In general, $\DEF(\Ell, T)$ is not closed  under enough finite left limits to be completed to a pre-topos. In this section we describe a completion process for metric logical categories which produces what in essence is the largest logical category in which $\R$ can be conservatively embedded. 

In order to describe the allowable limits and co-limits that we consider, we need some data.  Fix the following:
\begin{enumerate}
\item a metric logical category $\R$ with language and theory $(\Ell_\R,T_\R)$,
\item a small category $\Lambda$ and a diagram $D:\Lambda \rightarrow \R$,
\item an expansion of the language $\Ell_\R$ to a metric language $\Ell_\Lambda$ which includes a new sort $\SS$ and for every object $a$ in $\Lambda$, a function symbol $\pi_a:\SS \rightarrow \SS_{D(a)}$; in the case of co-limits, $\pi_a:\SS_{D(a)} \rightarrow \SS$, and 
\item $T_D$, a theory in the language $\Ell_\Lambda$ containing $T_\R$.
\end{enumerate}
With all of this data then, we say that $(D,T_D)$ is an axiomatizable cone (or co-cone) if whenever $\M$ satisfies $T_D$ then $(\SS(\M),\{\pi_a : a \in \text{Obj}(\Lambda)\})$ is a cone (or co-cone) for $ID$ where $I:\R \rightarrow \Metric$ is the logical functor associated to the $T_\R$-model $M\restriction_{\Ell_\R}$.  We say that such an $\M$ is a $\Sigma$-allowable cone (or co-cone).
 
If $(D,T_D)$ is an axiomatizable cone (co-cone),  we will call $(D,T_D)$ an axiomatizable limit (co-limit) if
\begin{enumerate}
\item there is a language $\Ell$, containing $\Ell_\R$ and an $\Ell$-theory $T$ such that the forgetful functor $F:\Mod(\Ell, T) \rightarrow \Mod(\Ell_\R, T_\R)$ is an equivalence of categories and
\item for any model $\M$ of $T$, $(\SS(\M),\{\pi_a : a \in \text{Obj}(\Lambda)\})$ is a limit (co-limit)  cone for $ID$ where $I:\R \rightarrow \Metric$ is the logical functor associated to the $T_\R$-model $F(\M)$ among all the $\Sigma$-allowable cones (co-cones).
\end{enumerate}
We say that $T$ axiomatizes the limit or co-limit for the axiomatizable cone $(D,T_D)$.
Some examples are definitely in order.

\begin{example} These examples correspond to each of the closure conditions given in the previous section for the creation of $T^\eq$.
\begin{enumerate}
\item Consider the case of the limit diagram for a countable product: the category $\Lambda$ is just a countable category with only the identity arrows.  If $\R$ is any metric logical category, $D:\Lambda \rightarrow \R$ is essentially just a choice of countably many objects from $\R$.  In this case, $T_D$ will just be $T_\R$.  If we have a $\Metric$-model of $T_\R$, that is a model of $T_\R$ with one additional sort and functions from that sort to countably many metric spaces chosen by $D$, this will be a cone for $D$.  The axioms listed for closure under countable products (see \ref{Ex.axioms} (\ref{Ex.ctble-prod})) in the construction of $T^\eq$ represent the theory $T$ asked for in the definition of axiomatizable limit and so countable products represent an example of axiomatizable limits.

\item Now suppose that $\varphi$ is a definable set for some sort $\SS_a$ corresponding to an object $a$ in $\R$, a metric logical category.  $\Lambda$ in this case is just a single object category with the identity map and $D$ will pick out the single sort $a$.  There is a function symbol $\pi_a \in \Ell_\Lambda$ from the new sort $S$ to $S_a$ and the only addtional axiom in $T_D$ beyond $T_\R$ is the sentence $\forall_\var{x} \varphi(\pi_a(\var{x}))$ which asserts that the range of $\pi_a$ is contained in the zero set of $\varphi$.  Any $\Metric$-model of $T_D$ will interpret $\pi_a$ as a map whose image is contained in the zero-set of $\varphi$.  Having a separate sort for the zero-set of $\varphi$ would axiomatize this limit and that is exactly what the axioms given in \ref{Ex.axioms} (\ref{Ex.def-sets}) do.

\item Let us do an instance of an axiomatizable co-limit. The most important one for us is the case of canonical parameters (\ref{Ex.axioms} (\ref{Ex.can-para})).  As with definable sets, $\Lambda$ will just be the trivial one point category.  Suppose that $D$ picks out an object $a$ from a metric logical category $\R$. The language $\Ell_\Lambda$ will have a new sort symbol $\SS$ and a function symbol $\pi_a:\SS_a \rightarrow \SS$.  Fix a formula $\varphi(\var{x},\var{y})$ from $\Ell_\R$ where $\var{y}$ is a variable from $\SS_a$, The only axiom beyone $T_\R$ in $T_D$ is
\[
d(\pi_a(\var{y}),\pi_a(\var{y}')) \leq \forall_\var{x} |\varphi(\var{x},\var{y}) - \varphi(\var{x},\var{y}')|.
\]
Any $\Metric$-model $\M$ of $T_D$ will have a metric space $\SS(\M)$ and map $\M(\pi_a)$ to the metric space $\SS_a(\M)$.  The axiom above will guarantee that $\M(\pi_a)$ factors through 
$\SS(\M)/\sim$ 
where $\sim$
 is the pseudo-metric defined by $\forall_\var{x} | \varphi(\var{x},\var{y}) - \varphi(\var{x},\var{y}') |$ which guarantees that the sort for the canonical parameters of $\varphi$ realizes the axiomatizable co-limit.
\item The case of finite unions (\ref{Ex.axioms}(\ref{Ex.unions})) is easy to handle and so we leave it to the reader.
\end{enumerate}
\end{example}

\begin{definition} Suppose that $\R$ is a metric logical category.
\begin{enumerate}
\item For an axiomatizable cone (or co-cone) $(D,T_D)$ we say that $\R$ has an axiomatizable limit for $(D,T_D)$ if $T_\R$ axiomatizes the limit (or co-limit) for this cone (or co-cone). That is, there is a sort in $\Ell_\R$ and the necessary connecting maps which acts as the necessary limit or co-limit.
\item A metric logical category  $\R$ is called a {\em metric pre-topos} if it has all axiomatizable limits and colimits. 
\end{enumerate}
\end{definition}

Notice that it is immediate from the definition that $\Metric$ is a metric pre-topos.

\begin{theorem}
Let $\R$ and $\S$ be metric logical categories, and let $I:\R\to\S$ be a logical functor. Then $I$ preserves all axiomatizable limits and colimits that exist in $\R$.	
\end{theorem}

\begin{definition}
Suppose that $\R$ is a metric logical category.  We define $\P(\R)$ to be the metric logical category $\DEF(\Ell^\eq_\R,T^\eq_\R)$.
\end{definition}

We record that

\begin{theorem}\label{thm:The-forgetful-functor-is-an-equivalence}
The forgetful functor $I^*:\Mod(\P(\R))\to\Mod(\R)$ is an equivalence of categories.
\end{theorem}

We can now restate Theorem \ref{thm:weak-conceptual-completeness} in the context of metric logical categories, and conclude that $\P(\R)$ is the largest category that is a conservative expansion of $\R$.

\begin{theorem}\label{thm:maximal-conservative-expansion-existence}
Let $\R$ be a metric logical category, and consider the category $\P(\R)$ and the logical functor $I:\R\to\P(\R)$. If $\S$ is a logical category, and $J:\R\to\S$ is a logical functor such that $J^*$ is an equivalence of categories, then there is a conservative logical functor $K:\S\to\P(\R)$ such that $KJ=I$.
\end{theorem}

The following theorem is an immediate consequence of the construction of $(\Ell^\eq, T^\eq)$, and the definition of $\P(\R)$ in the previous sections.


\begin{theorem}\label{thm:T-eq-is-a-pre-topos}
For every metric logical category $\R$, $\P(\R)$ is a metric pre-topos. Furthermore, the map $I:\R\to\P(\R)$ is universal in the sense that if $\T$ is any metric pre-topos, and $J:\R\to\T$ is a logical functor, then there is a logical functor $K:\P(\R)\to\T$ such that $KJ=I$.	
\end{theorem}


\begin{theorem}
Let $\T$ be a metric pre-topos, and let $I:\T\to\S$ be a logical functor. If $I^*$ is an equivalence of categories, then so is $I$.
\end{theorem}

\begin{proof}
Given a logical functor $I:\T\to\S$ with the property that $I^*$ is an equivalence of categories, then by Theorem \ref{thm:maximal-conservative-expansion-existence}, there is a logical functor $J:\S\to\P(\T)$. Since $\T$ is a pre-topos, there is an equivalence of categories $E:\T\to\P(\T)$, so that $EJ:\S\to \T$. A straightforward calculation shows that $EJ$ is inverse to $I$, which implies $I$ is an equivalence of categories.
\end{proof}

\begin{corollary}\label{thm:pre-topos-equiv-of-cats}
Every small pretopos is of the form $\DEF(\Ell^\eq, T^\eq)$
\end{corollary}

\nocite{model-theory, cc-continuous-logic, basic-concepts-of-enriched-category-theory, categories-for-the-working-mathematician, metric-spaces-generalized, ultraproducts-in-analysis, ttt}

\bibliographystyle{plainnat}
\bibliography{bibliography}

\end{document}